\documentclass[a4paper, svgnames, 12pt]{amsart}
\usepackage[english]{babel}
\usepackage{amsmath, amsthm, amssymb, hyperref, color}
\usepackage[capitalize,noabbrev]{cleveref}
\usepackage{graphicx}
\usepackage{caption}
\usepackage{subcaption}
\usepackage{mathtools}
\usepackage{enumerate}
\usepackage{verbatim}
\usepackage{tikz}
\usepackage{tikz-cd}
\usepackage{enumitem}
\usepackage{titlesec}
\usepackage{hyperref}
\hypersetup{ 
	colorlinks = true,
	allcolors = red,
}

\usepackage[a4paper]{geometry}
\titleformat{\section}
{\normalfont\Large\bfseries}{\thesection}{1em}{}[{\titlerule[0.8pt]}]

\newtheorem{theorem}{Theorem}[section]
\newtheorem{proposition}[theorem]{Proposition}
\newtheorem{lemma}[theorem]{Lemma}
\newtheorem{corollary}[theorem]{Corollary}

\theoremstyle{definition}
\newtheorem{definition}[theorem]{Definition}

\newtheorem{remark}[theorem]{Remark}

\newtheorem{example}[theorem]{Example}

\DeclareMathOperator{\dist}{dist}

\DeclareMathOperator{\PZ}{\mathrm{PZ}}

\DeclareMathOperator{\SL}{SL}

\DeclareMathOperator{\End}{\mathrm{End}}

\DeclareMathOperator{\Id}{Id}
\DeclareMathOperator{\Jac}{Jac}

\DeclareMathOperator{\B}{B}

\newcommand{\Beta}{\mathbb{B}}

\newcommand{\RR}{\mathbb{R}}
\newcommand{\QQ}{\mathbb{Q}}

\newcommand{\ZZ}{\mathbb{Z}}

\newcommand{\Ocal}{\mathcal{O}}

\newcommand{\Bcal}{\mathcal{B}}

\newcommand{\msc}[1]{\href{https://zbmath.org/classification/?q=#1}{#1}}

\numberwithin{equation}{section}

\makeatletter
\renewcommand\subsection{\@startsection{subsection}{2}%
	\z@{.5\linespacing\@plus.7\linespacing}{-.5em}%
	{\normalfont \bfseries}}
\makeatother

\makeatletter
\@namedef{subjclassname@2020}{%
	\textup{2020} Mathematics Subject Classification}
\makeatother

\begin{document}
	
	\title{Bolytrope Orders}
	
	\author[Y.~El Maazouz]{Yassine El Maazouz}
	\address[Yassine El Maazouz]{UC Berkeley}
	\email{yassine.el-maazouz@berkeley.edu}
	
	\author[G.~Nebe]{Gabriele Nebe}
	\address[Gabriele Nebe]{RWTH Aachen University}
	\email{gabriele.nebe@rwth-aachen.de}
	
	\author[M.~Stanojkovski]{Mima Stanojkovski}
	\address[Mima Stanojkovski]{RWTH Aachen University and 
		MPI-MiS Leipzig}
	\email{mima.stanojkovski@rwth-aachen.de}
	
	
	
	\date{}
	
	\maketitle
	
	\begin{abstract}
		Bolytropes are bounded subsets of an affine building that consist of all points that have distance at most $r$ from some polytrope. We prove that the points of a bolytrope describe the set of all invariant lattices of a bolytrope order, generalizing the correspondence between polytropes and graduated orders.
	\end{abstract}

\section{Introduction}
	
The work of this paper finds its purpose within the framework of a larger joint project involving the three authors, 
Marvin Anas Hahn, and Bernd Sturmfels,  
and regarding the investigation of the interplay
between orders $\Lambda$
over discrete valuation rings (in split simple algebras)
 and certain bounded convex subsets $Q$ of affine buildings.
 The set of $0$-simplices, ${\mathcal B}^0$, of the building 
is in bijection with the maximal orders \cite{AbramenkoNebe,Garret}. 
 Closed orders are 
 intersections of finitely many maximal orders. These are 
exactly the Plesken-Zassenhaus orders $\Lambda = \PZ (\mathcal{L})$ of the finite subsets $\mathcal{L}$ of the building; cf.\ \Cref{def:PZ} and Proposition \ref{degfin}. 
Any order $\Lambda $ defines a 
bounded convex set $Q(\Lambda )\subset {\mathcal B}^0$ corresponding to 
the maximal orders that contain $\Lambda $. 
A set of the form $Q(\Lambda )$ is called 
Plesken-Zassenhaus-closed ($\PZ$-closed, for short). 
The $\PZ $-closed sets that are contained in one apartment
are exactly the polytropes 
and the corresponding $\PZ$-orders are the graduated
 orders; cf.\ 
 \cite{EHNSS/21, shemanske} and references therein. 
In this paper, we define the new class of bolytrope orders $\Lambda$ for which $Q(\Lambda)$ is a bolytrope. 
In turn, a bolytrope is defined as the set of all elements in ${\mathcal B}^0$  that have distance at most $r$ from a polytrope, the associated central polytrope.
The special case of ball orders, i.e.\ bolytrope orders where the central polytrope consists of one point, is treated in \Cref{ballorders}. 

The word  ``bolytrope'' is of our own invention and is the fusion of the words ``ball'' and ``polytrope''. Polytropes are the tropical analoga of 
polytopes (see for instance \cite{Jos22}) and describe the 
sets of lattices that are invariant under graduated orders. A bolytrope order is the 
intersection of a ball order and a gradudated order (see \Cref{lem:bolytrope1}) and a bolytrope is the set of invariant lattices
of a bolytrope order. 

One of the main results of this paper is \Cref{th:bolytropes}, stating that bolytropes and bolytrope orders are 
$\PZ $-closed and closed, respectively. 

Closed orders in quaternion algebras have been extensively studied 
in the context of class groups, 
Brandt matrices, and Hecke operators, cf.\ 
 \cite{Arenas3,Arenas4,Arenas2,Arenas1,Tu/11}. 
In Section \ref{d=2}, we apply our results to reprove and extend the work from \cite{Tu/11}, showing in particular that all closed split quaternion orders 
 are bolytrope orders.

In \Cref{Radid}, we present our main tool: the radical idealizer chain of an order. This allows for an inductive procedure to handle bolytrope orders as explained in \Cref{lem:bolytrope3} and also shows 
that the central polytrope $Q$ is uniquely determined by the bolytrope  $B$; its Plesken-Zassenhaus order $\PZ(Q)$ is the 
first term in the radical idealizer chain applied to
the bolytrope order $\PZ(B)$ that happens to be a graduated order. 

\subsection{Notation}

Throughout the paper let $K$ be a discretely valued field
with  
valuation ring $\Ocal_K$, unifomizer $\pi$, and maximal ideal $\mathfrak{m}_K=\Ocal_K\pi \neq \{ 0 \}$. 
If $K$ is commutative, which we assume henceforth, 
there is no need for $K$ to be complete: in particular, $K=\QQ $ with
some $p$-adic valuation is allowed. 
 Let, moreover, $d$ be a positive integer. We write ${\bf 1}$ for the vector $(1,\ldots, 1)\in\ZZ^d$ and $J_d$ for the matrix, in $\ZZ^{d\times d}$, with zeros on the diagonal and ones elsewhere. 

\section{Graduated orders and apartments} 

 An \emph{$\Ocal_K$-lattice} (or simply lattice) in $K^d$ is a free $\Ocal_K$-submodule of maximal rank $d$. The (\emph{homothety}) \emph{class} of a lattice $L$ in $K^d$ is
 \[
 [L]=\{cL \mid c\in K\setminus\{0\}\}=\{\pi^nL \mid n\in\ZZ\},
 \]
 while $\End_{\Ocal_K}(L)$ denotes the endomorphism ring of $L$ as an $\Ocal_K$-module. Note that any two homothetic lattices have the same endomorphism ring.
 
 An \emph{$\Ocal _K$-order $\Lambda$} in the matrix ring $K^{d\times d}$ is an 
 $\Ocal _K$-lattice that is also a ring,
 i.e.\ $\Lambda $ is multiplicatively closed and contains the 
 identity element $\Id _d$ of $K^{d\times d}$. If $\Lambda$ is an order in $K^{d\times d}$, then a {\em $\Lambda$-lattice} is a lattice in $K^d$ that is also a $\Lambda$-module.

\subsection{Graduated orders and polytropes} 

In the present section, we define graduated orders following \cite{Ple83} and collect some related results from \cite{EHNSS/21,Ple83}. 

\begin{definition}\label{def:graduated}
	An $\Ocal _K$-order $\Lambda $ in $K^{d\times d}$ 
	is called {\em graduated} if $\Lambda $ contains a
	complete set of orthogonal primitive idempotents 
	$\epsilon _1,\ldots , \epsilon _d$ of $K^{d\times d}$.
\end{definition}

The primitive idempotents of $K^{d\times d}$ are exactly the
projections onto 1-dimensional subspaces of $K^d$, so each set
$\{ \epsilon _1,\ldots , \epsilon _d \}$ as in Definition \ref{def:graduated}
defines a {\em frame}
$$K^d = \epsilon _1 K^d \oplus \ldots \oplus \epsilon _d K^d = K e_1  \oplus \ldots \oplus K e_d  $$
i.e. a decomposition of $K^d$ as a direct sum of 1-dimensional subspaces.
In any frame basis $(e_1,\ldots , e_d)$ the idempotents are diagonal
matrices with exactly one entry 1 on the diagonal.
The projection onto the $ij$-matrix entry
$\epsilon _i \Lambda \epsilon_j $ is an $\Ocal _K$-submodule
of $\epsilon _i K^{d\times d} \epsilon _j \cong K$. 
Hence, writing matrices with respect to the frame basis 
$(e_1,\ldots ,e_d)$, there exists a matrix $M = (M_{ij}) \in \ZZ ^{d\times d}$ such that  the graduated order $\Lambda $ is 
of the form 
\begin{align*}
\Lambda (M) &= \{ X=(X_{ij})\in K^{d\times d} \mid X_{ij} \in {\mathfrak m}_K^{M_{ij}} 
	\mbox{ for all } i,j = 1,\ldots ,d \}. 
\end{align*}
	The matrix $M$ is called the {\em exponent matrix} of $\Lambda$.

	\begin{remark}
	When we state that a given order $\Lambda $ is contained in some graduated 
	order $\Lambda (M)$ we always mean that there exists a suitable basis such that this
	graduated overorder of $\Lambda $ has the form $\Lambda(M)$. How to find such a basis 
	is usually indicated in the proofs. 
	\end{remark}
The fact that $\Lambda=\Lambda(M)$ is a ring is equivalent to having, for all $i,j,k\in\{1,\ldots, d\}$, that
\begin{equation}\label{eq:Pd}
    M_{ii} = 0\ \textup{ and }\ M_{ij} + M_{jk} \geq M_{ik}.
\end{equation}
With the {\em polytrope region} $\mathcal{P}_d$ as defined in \cite[Section~4]{EHNSS/21}, we have that \eqref{eq:Pd} is equivalent to the condition $M\in \mathcal{P}_d\cap \ZZ^{d\times d}$. Putting 
\begin{align*}
  {\mathcal P}_d (\ZZ )
   = \{ M \in \ZZ ^{d\times d}\mid M_{ii} = 0, \ M_{ij} + M_{jk} \geq M_{ik}  \mbox{ for all } i,j,k =1,\ldots ,d \},  
\end{align*}
we can see that \eqref{eq:Pd} is equivalent to $M\in\mathcal{P}_d(\ZZ)$.

\begin{remark}\label{exponentvector} 
		For $M\in {\mathcal P}_d(\ZZ)$, the 
 $\Lambda (M)$-lattices $L$ are of the form
		$L= \bigoplus _{i=1}^d \epsilon _i L $ and hence 
		there exists $u=(u_1,\ldots , u_d)\in \ZZ^d$ such that 
		\[
		L = L_u := {\mathcal O}_K \pi^{u_1} e_1 \oplus \ldots \oplus {\mathcal O}_K \pi ^{u_d}  e_d .
		\]
		The  tuple $u$ is called the
		{\em exponent vector} of the lattice $L$. Moreover, 
		$L=L_u$ is a $\Lambda(M)$-lattice if and only if, for any choice of $1\leq i,j\leq d$, one has 
		$M_{ij}+u_j \geq u_i$ and two $\Lambda(M)$-lattices $L_u$ and $L_v$ are isomorphic if and only if $u-v\in \ZZ{\bf 1}$.
		Put 
		\begin{align*}
		Q_M:= \{[u]\in \RR^d/\RR{\bf 1} \mid M_{ij}+u_j \geq u_i \mbox{ for all } i,j =1,\ldots ,d\}.
		\end{align*}
		Then $Q_M$ is a  {\em polytrope} and the integral points
		of $Q_M$ parametrize the $\Lambda(M)$-stable lattices in $K^d$.
	\end{remark}

\subsection{Buildings and apartments}

In line with the content of this paper, we define the affine building of $\SL_d(K)$ via its lattice class model and refer the interested reader to \cite{AbramenkoBrown} for the more general description.

\begin{definition}\label{def:building}
The \emph{affine building} $\Bcal _d(K)$ is an infinite simplicial
complex such that  
\begin{enumerate}[label=$(\arabic*)$]
	\item the vertex set is $\mathcal{B}_d^0(K)=\{ [L] \mid L \mbox{ is an $\Ocal _K$-lattice in } K^{d} \}. $
    \item $\{ [L_1], \ldots , [L_s] \} $ is a simplex
in $\Bcal _d(K)$ if and only if, up to permutation of the indices and choice of representatives, one has $L_1 \supset L_2 \supset \cdots \supset L_s \supset \pi L_1$.
\end{enumerate}
An {\em apartment} of $\mathcal{B}_d(K)$ is any subset $\mathcal{A}(E)$ of $\mathcal{B}^0_d(K)$ of the form 
$${\mathcal A}(E) := \{ [L_u] \mid 
u=(u_1,\ldots , u_d) \in \ZZ ^d \}$$
where $E$ is a frame $K^d=\epsilon _1 K^d \oplus \ldots \oplus \epsilon _d K^d=Ke_1\oplus\ldots\oplus Ke_d$ of $K^d$.
\end{definition}

Strictly speaking, ${\mathcal A}(E)$ consists only of the set 
of 0-simplices in the apartment.
Note that, whereas the exponent vector $u$ depends 
on the choice of the basis $(e_1,\ldots ,e_d)$, the whole 
apartment ${\mathcal A}(E) $ only depends on the frame 
or the corresponding set $\{ \epsilon_1,\ldots , \epsilon _d \}$ 
of orthogonal primitive  idempotents. 
An explicit choice of the basis gives an identification of 
the lattice classes  $[L_u] \in {\mathcal A}(E)$ 
with the integral points $[u] \in \ZZ^d/ \ZZ {\bf 1} $
in the $(d-1)$-dimensional space $ \RR ^d /\RR {\bf 1}$.

 \subsection{Closed orders} \label{sec:closed} 


	\begin{definition}
		Let $\Lambda $ be an order in $K^{d\times d}$.
		Then 
		$$Q(\Lambda ) := \{ [L ] \in {\mathcal B}^0_d(K) \mid \Lambda L = L \} $$
denotes the set of homothety classes of 
		$\Lambda $-lattices in $K^d$.
	The order $\Lambda $ is called 
		{\em closed} if 
\[
\Lambda =\bigcap_{[L]\in Q(\Lambda)}\End_{\Ocal_K}(L).
\]
\end{definition}

Notice that the closed orders are exactly the ones 
that are determined by their sets of invariant lattices. This is not the case in general as the following example shows. 

\begin{example}\label{ex:nonclosed}
Let $M:={\tiny
\begin{pmatrix}
0 & 0 \\ 1 & 0
\end{pmatrix}}\in\mathcal{P}_d(\ZZ)$.
	Then  $\Lambda (M)$ is a graduated order with $Q(\Lambda(M)) = \{ [L_{(0,1)}], [L_{(0,0)}] \} $. Let $\Lambda := \{ X\in \Lambda(M) \mid X_{11} \equiv X_{22} \bmod \pi \} $. Then $\Lambda $ is an order in $K^{2\times 2} $ satisfying $Q(\Lambda) = Q(\Lambda(M) )$. It follows that $\Lambda $ is not closed. 
\end{example}

 \section{The distance, balls and bolytropes} 

In this section, we define a distance on $\mathcal{B}_d^0(K)$ and use it to define balls and bolytropes in the building $\mathcal{B}_d(K)$. Balls are a special type of bolytropes and bolytropes can be thought of as balls ``around polytropes''.

\subsection{The distance} 
The work in this paper heavily depends on the following notion of 
distance on $\mathcal{B}_d^0(K)$. 

\begin{definition}\label{def:dist}
	Let $[L_1], [L_2] \in {\mathcal B}_d^0(K)$ be two 
	homothety classes of lattices. 
	Then 
$$\dist([L_1],[L_2]) := \min \{ s \mid 
	\mbox{ there are } L_1'\in [L_1], L_2' \in [L_2] \mbox{ with } 
	\pi ^{s } L_1' \subseteq L_2' \subseteq L_1' \} .$$
	For a subset ${\mathcal L}  \subseteq {\mathcal B}^0_d(K)$, we put $\dist([L],{\mathcal L} ) := \min \{\dist ([L],[L']) \mid [L']\in\mathcal{L} \}$.
	The set ${\mathcal L}$ is called {\em bounded}, if 
	$\sup \{ \dist ([L],[L'] ) \mid [L], [L'] \in {\mathcal L} \} $
	is finite. 
	\end{definition}

The following result ensures that $\dist$ is in fact a distance.

\begin{lemma}\label{lem:dist-welldef}
    The map $\dist:\mathcal{B}_d^0(K)\times \mathcal{B}_d^0(K)\rightarrow \ZZ$ is a distance on $\mathcal{B}_d^0(K)$. 
\end{lemma}

\begin{proof}
We check that the defining properties of a distance hold. For this, let $[L_1],[L_2]\in\mathcal{B}_d^0(K)$ with $\dist([L_1],[L_2])=s$ and let $L_1',L_2'$ be as in \Cref{def:dist}. Then
	\begin{enumerate}[label=$(\arabic*)$]
		\item $\dist ([L_1],[L_2]) = 0 $ if and only $L_1' \subseteq L_2' \subseteq L_1'$, equivalently $[L_1]=[L_2]$. 
		\item If $\pi ^{s } L_1' \subseteq L_2' \subseteq L_1 $ then 
			$\pi ^{s } L_2' \subseteq \pi ^{s} L_1' \subseteq L_2' $, 
			so $\dist([L_1],[L_2]) = \dist ([L_2],[L_1]) $. 
		\item Let $[L_3]\in\mathcal{B}_d^0(K)$ and set  $s'=\dist([L_2],[L_3])$. Let, moreover $L_3'\in [L_3]$ and $L_2''\in[L_2]$ be such that $\pi^{s'}L_2''\subseteq L_3'\subseteq L_2''$. Write $L_2''=\pi^tL_2'$. Then 
		\[
		\pi^tL_1'\supseteq \pi^tL_2'\supseteq L_3'\supseteq \pi^{s'+t}L_2'\supseteq \pi^{s+s'+t}L_1'=\pi^{s+s'}(\pi^tL_1'),
		\]
		yielding that 
		$ \dist([L_1],[L_2]) + \dist([L_2],[L_3]) \geq \dist ([L_1],[L_3])$.
	\end{enumerate}
	The choices of $[L_1], [L_2], [L_3]$ being arbitrary, the proof is complete.
\end{proof}

Thanks to the elementary divisor theorem for modules over PIDs, we know that any two lattices in $K^d$ have compatible bases, i.e.\ 
for any two lattice classes $[L_1]$ and $[L_2]$, there is always an apartment containing both.
So, to compute their distance, we may choose a frame basis $(e_1,\ldots ,e_d)$ of 
$K^d$, so that $L_1 = L_{(0,\ldots , 0)}$ and $L_2= L_{(u_1,\ldots , u_d )}$ with 
$u_1\geq \ldots \geq u_d$. With this choice, we obtain that $\dist ([L_1],[L_2]) = u_1-u_d $.

\begin{remark} 
	The distance between lattice classes $[L_u]$ and 
	$[L_v]$ in the same apartment $\mathcal{A}(E)$ is given by
	$$\dist([L_u],[L_v])=\max _{1\leq i\leq d} (v_i-u_i) - \min _{1\leq j\leq d} (v_j-u_j) .$$
In particular, any bounded subset of an apartment is finite. For a connection to tropical geometry, see for instance \cite[Section 5.3]{Jos22}.
\end{remark}

	Note that the distance from Definition \ref{def:dist} 
coincides with the $1$-skeleton distance on $\mathcal{B}_d^0(K)$, as the following result shows. For $L$ and $L'$ lattices with $\pi L\subset L'\subset L$, write $([L],[L'])$ for the $1$-simplex with ends $[L]$ and $[L']$. 

\begin{lemma}
Let $[L_1],[L_2]\in\mathcal{B}_0^d(K)$ be distinct and set $s=\dist([L_1],[L_2])$.  Then $s>0$,
\begin{enumerate}[label=$(\arabic*)$]
    \item 
 there exist $[L_1] =[X_0],[X_1] \ldots, [X_{s-1}], [X_s] = [L_2] \in\mathcal{B}_d^0(K)$ such that   
$ ([X_{i-1}],[X_i]) $ are $1$-simplices for all $1\leq i\leq s$, and 
    \item there is no shorter sequence connecting $[L_1]$ and $[L_2]$ in the 1-skeleton of ${\mathcal B}_d(K)$.
\end{enumerate}
\end{lemma}

\begin{proof}
The number $s$ is positive as a consequence of \Cref{lem:dist-welldef}. Without loss of generality, assume that $\pi^s L_1 \subseteq L_2 \subseteq L_1$ and put $X_1:= \pi L_1 + L_2 $. Then $\pi L_1 \subseteq X_1 \subseteq L_1 $ and so 
	$([L_1],[X_1])$ is a 1-simplex in ${\mathcal B}_d(K)$. 
	For $i= 2, \ldots ,s$, put $X_i := \pi X_{i-1} + L_2 = \pi ^{i} L_1 + L_2$. Then $X_s=L_2$ and all $([X_{i-1}],[X_i])$
	are $1$-simplices in the building. We have proven (1), while  (2) follows from the triangle inequality 
	and the fact that  two lattice classes in a $1$-simplex have distance at most $1$.
\end{proof}

\subsection{Balls and bolytropes}

	\begin{definition} \label{def:ball} 
		Let ${\mathcal L} $ be a bounded 
		subset of ${\mathcal B}_d^0(K)$. Then  the {\em closed ball of radius $r$ and center ${\mathcal L}$} is
		$$\B_r({\mathcal L}) := \{ [L] \in {\mathcal B}_d^0(K) \mid 
		\dist ([L],{\mathcal L} ) \leq r \} .$$ 
		If ${\mathcal L} = \{ [L] \} $ consists of one element only, then 
		$$ \B_r([L]) := \B_r({\mathcal L})$$ 
		is the {\em ball with center $[L]$ and radius $r$}.
		If ${\mathcal L} = Q(\Lambda (M))$, then 
		$$\B_r(M):= \B_r (Q(\Lambda (M)))$$ is called the {\em bolytrope with
		center $Q(\Lambda(M))$} and radius $r$. 
	\end{definition}

	In particular, the ball $\B_r([L])$ consists of all lattice classes $[L']$
	that are represented by some lattice $L'$ such that $\pi ^r L \subseteq L' \subseteq L$. We close the section by computing the intersection of a bolytrope with an apartment. Recall that  $J_d \in {\mathcal P}_d(\ZZ ) $ is the matrix with all $1$s outside of the main diagonal.

\begin{lemma} \label{ballapp} 
	Let ${\mathcal A}$ be an apartment containing $Q(\Lambda(M))$. Then  $$\B_r(M) \cap {\mathcal A} = Q(\Lambda({M+rJ_d})).$$
\end{lemma}


\begin{proof}
Let $(e_1,\ldots,e_d)$ be a frame basis defining $\mathcal{A}$ and put $Q=Q(\Lambda (M+rJ_d))$. We will use \Cref{exponentvector} with respect to this basis.
Since $\pi^r \Lambda (M) \subseteq \Lambda (M+rJ_d) \subseteq \Lambda (M)$, we have the inclusion $Q\subseteq \B_r(M) $. Now we show the other inclusion. Let $[L_u]$ in $\mathcal{A}$ be of distance at most $r$ from some lattice $[L_v] \in Q(\Lambda(M))$. Suppose that $[L_u] \not \in Q(\Lambda(M + rJ_d))$. This means that there exist $1 \leq i \neq j \leq d$ such that $u_i - u_j > M_{ij} + r$. However, since $[L_v] \in Q(\Lambda(M))$, we have $v_i - v_j \leq M_{ij}$ and hence $u_i - u_j > v_i - v_j + r$. In other words
\[
(u_i - v_i) - (u_j - v_j) > r, \mbox{ so } 
 \dist ([L_u] , [L_v ]) > r .
\]
This is a contradiction and so the proof is complete. 
\end{proof}

\subsection{Plesken-Zassenhaus closed sets.}

We have seen that closed orders are determined by the collection of their stable lattices; such sets are thus of fundamental importance for the study of closed orders.

\begin{definition}
	A subset ${\mathcal L} $ of ${\mathcal B}_d^0(K)$ is called 
	\emph{$\PZ$-closed} if ${\mathcal L} = Q(\Lambda )$ for some order
	$\Lambda $. 
\end{definition}

For the study of $\PZ$-closed subsets it clearly 
suffices to consider 
closed orders $\Lambda$. Note that the bijection $\Lambda  \leftrightarrow  Q(\Lambda )$
is a Galois correspondence between 
\[
\{\textup{ closed orders in } K^{d\times d}\  \}
\longleftrightarrow
\{
\textup{ $\PZ$-closed subsets of } \mathcal{B}_d^0(K)\ 
\}.
\]
As shown in \cite{EHNSS/21}  (see also \Cref{polytrope}), 
the $\PZ $-closed subsets of one apartment $\mathcal{A}$ are exactly the finite and convex subsets of $\mathcal{B}_d^0(K)$, i.e.\ the
polytropes. In general, being bounded and convex is a necessary but 
not sufficient condition for a subset of $\mathcal{B}_d^0(K)$ to be closed. 

\begin{proposition}\label{prop:convS}
Let $\Lambda$ be an order in $K^{d\times d}$. Then  $Q(\Lambda)$ is a 
	non-empty bounded convex subset of $\mathcal{B}_d^0(K)$. 
\end{proposition}

\begin{proof} 
As any order is contained in a maximal order, 	there is some 	maximal order $\Gamma $, with 	$\Lambda \subseteq \Gamma $.
Both lattices $\Lambda $ and $\Gamma $ have full rank in
$K^{d\times d}$, so there is $r\in \ZZ_{\geq 0}$ such that $\pi^r \Gamma \subseteq \Lambda 
        \subseteq \Gamma $.
    If $[L]$ is the unique class of $\Gamma $-lattices, then
    $[L] \in Q(\Lambda )$ and hence $Q(\Lambda)$ is not empty.
 Moreover, all lattice classes in $Q(\Lambda )$ have a representative 
                        between $L$ and $(\pi^r \Gamma ) L = \pi^r L$, so $Q(\Lambda )$ is
                        contained in the ball of radius
                        $r$ around $[L]$. In particular, $Q(\Lambda )$ is
                        bounded.
To see convexity, let $[L'],[L''] \in Q(\Lambda )$.
Then there is an apartment containing both lattice classes, so
        $\Gamma ' := \End_{\Ocal _K}(L') \cap \End_{\Ocal _K}(L'') $ is a graduated order containing $\Lambda $.
        But then the convex set $Q(\Gamma ') \subseteq Q(\Lambda )$ contains both lattice classes $[L']$ and $[L'']$, and, 
  $[L']$ and $[L'']$ being arbitrary, $Q(\Lambda )$ is convex.
\end{proof}	

\begin{remark}
Let $\Lambda$ be an order in $K^{d\times d}$ and let $\mathcal{A}$ be an apartment in $\mathcal{B}_d(K)$ such that $Q(\Lambda ) \cap {\mathcal A} \neq \emptyset $. 
Then $$Q(\Lambda ) \cap {\mathcal A} = Q(\Gamma )$$ 
for a unique graduated overorder $\Gamma $ of 
$\Lambda $.
Indeed, if $\mathcal{A} = {\mathcal A}(E)$ 
and ${\mathcal E}= \{ \epsilon_1,\ldots , \epsilon _d \}$  is the set of projections on the
frame $E$, then there are 
only finitely many maximal overorders of 
$\Lambda $ that contain ${\mathcal E}$. 
Their intersection is the desired graduated order
$\Gamma $. 
\end{remark}

\subsection{The degree of a closed order}

\begin{definition}\label{def:PZ}
Let ${\mathcal L} $ be a bounded subset of ${\mathcal B}_d^0(K)$. The
	{\em Plesken-Zassenhaus order} associated to $\mathcal{L}$ is
	$$\PZ ({\mathcal L}) := \bigcap _{[L]\in {\mathcal L} } \End _{\Ocal _K} (L) .$$ 
\end{definition}

\begin{proposition}
	The Plesken-Zassenhaus order $\PZ({\mathcal L})$ of 
	a bounded subset ${\mathcal L}$ of ${\mathcal B}_d^0(K)$ is an $\Ocal _K$-order in $K^{d\times d}$
\end{proposition}

\begin{proof}
	Put $\Lambda = \PZ ({\mathcal L})$.
	Then $\Lambda $ is an $\Ocal _K$-module that is closed under multiplication and contains $\Id _d$. 
	It remains to show that $\Lambda $ is of full rank in $K^{d\times d}$. 
	As ${\mathcal L}$ is bounded, there are $[L] \in {\mathcal L}$ and $r\in \ZZ_{\geq 0}$ such that 
	${\mathcal L} \subseteq \B _r([L]) $. For the maximal order $\Gamma  = \End_{\Ocal _K} (L)$ 
	we hence have that $\pi ^r \Gamma  L' \subseteq L' $ for all $[L'] \in {\mathcal L}$.
	So $\pi ^r \Gamma  \subseteq \Lambda \subseteq \Gamma $ and, as $\pi ^r \Gamma $ contains a $K$-basis 
	of $K^{d \times d}$, the same is true for $\Lambda $.
\end{proof}

The following remark illustrates the notions introduced in Section \ref{sec:closed} 
and Definition \ref{def:PZ} for the special case of graduated orders.

\begin{remark} \label{polytrope} (\cite[Proposition 6 and 7, Corollary 9, Theorem 16]{EHNSS/21})
If $M\in {\mathcal P}_d(\ZZ)$, then the graduated order 
 $\Lambda(M)$ is  closed and 
	$$Q(\Lambda (M)) = \{ [L_u ] \mid u\in \ZZ ^d,\  u+\RR{\bf 1} \in Q_M \}$$
is a finite set which we can identify  with the integral points of the polytrope $Q_M$.  
	 Moreover, the projective $\Lambda(M)$-lattices 
	 are given by the columns of $M$ in the following way:  if $M^{(1)},\ldots M^{(d)}$ denote the columns of $M$, then, for each projective $\Lambda(M)$-lattice $L$, there exists $i\in\{1,\ldots,d\}$ such that $L$ is homothetic to
	 $$P_i := \Lambda (M) \epsilon_ i = L_{M^{(i)}}.$$ 
	 The polytrope $Q_M$ is the min-convex hull of
	 the set $\{ M^{(1)}+\RR {\bf 1}, \ldots , M^{(d)}+\RR {\bf 1} \} $ 
	and has  dimension  $\dim(Q_M) = \lvert \{ [P_1], \ldots , [P_d] \}  \rvert  - 1$.
	The order 
	 $\Lambda (M) = \PZ ([P_1], \ldots , [P_d] ) $ 
	 is the Plesken-Zassenhaus order of its projective lattices in
	 $K^d$. 
 \end{remark} 
 
 The next proposition shows that closed orders are always an intersection of finitely many maximal orders. 
 
 \begin{proposition}\label{degfin}
 Let ${\mathcal L} \subseteq {\mathcal B}^0_d(K)$ be bounded and let $\Lambda = \PZ ({\mathcal L})$ denote its Plesken-Zassenhaus order.
 Then there exists a finite subset $\{[L_1],\ldots , [L_n]\}$ of ${\mathcal L}$ such that $\Lambda = \PZ ([L_1],\ldots , [L_n]  )$. 
 \end{proposition} 
 
 \begin{proof}
 Choose $[L_1] \in {\mathcal L}$ arbitrarily and put $\Gamma = \End_{\Ocal _K}(L_1)$. 
 As ${\mathcal L}$ is bounded, there is $r\in \ZZ _{\geq 0} $ such that ${\mathcal L} \subseteq \B_r([L_1])$ and so
 $$\pi^r \Gamma \subseteq \Lambda \subseteq \Gamma .$$ 
 In particular, the $\Ocal _K$-module $\Gamma /\Lambda $ has finite composition length (at most the composition length $d^2r$ of $\Gamma /\pi^r\Gamma $).
 We proceed by induction on this composition length. 
 If $\Gamma = \Lambda$ then we are done, otherwise there is some $[L_2]\in {\mathcal L}$ such that $[L_2] \not\in Q(\Gamma )$. 
 Replace $\Gamma $ by $\Gamma \cap \End_{\Ocal _K}(L_2) = \PZ ( [L_1 ], [L_2 ] )$ to decrease the composition length 
 of $\Gamma /\Lambda $. After finitely many steps this process constructs the finite set $\{ [L_1],\ldots , [L_n] \}$ 
 with $\Lambda = \PZ ( [L_1],\ldots , [L_n]  )$.
 \end{proof}
 
For a closed order $\Lambda $, the minimal cardinality of a
set ${\mathcal L}$ 
such that $\Lambda = \PZ ({\mathcal L}) $ is hence an interesting invariant.

\begin{definition} 
	Let $\Lambda $ be a closed order. 
	Then  the \emph{degree} of $\Lambda $ is
	$$\deg(\Lambda ) := \min \{ \lvert  {\mathcal L}  \rvert  - 1 \mid \mathcal{L} \subseteq \mathcal{B}_d^0(K) \textup{ with } 
	\Lambda = \PZ ({\mathcal L}  ) \} .$$
\end{definition} 

Thanks to Proposition \ref{degfin}, any closed order is a finite intersection of maximal orders, so the 
degree of a closed order is always finite. 
The closed orders of degree $0$ are exactly the maximal orders and 
the ones of degree $1$ are certain graduated orders. 
In general, the degree of a 
graduated order $\Lambda (M)$ is equal to $\dim(Q_M)$, cf.\ Remark \ref{polytrope}.
In the coming sections, we will see that, for ball orders and bolytrope orders, the degree is always bounded from above by $d$, cf.\ \Cref{prop:balls,thm:bolystar}, though such bound need not always be sharp, cf.\ \Cref{ex:yassine}.

\section{The radical idealizer process} \label{Radid}

Let $\Lambda $ be an order in $K^{d\times d}$. In this section, we describe the radical idealizer chain of $\Lambda$, a construction that will be at the foundation of the proofs of our main results. 


\begin{definition} 
	Let $\Lambda$ and $L$ be an order and a lattice in $K^{d \times d}$, respectively. 
	\begin{itemize}
	    \item The {\em Jacobson radical} $\Jac(\Lambda )$ of $\Lambda$
is the intersection of all maximal left ideals of $\Lambda $.
\item The {\em idealizer} of $L$ is 
	$\Id(L) := \{ X\in K^{d\times d} \mid X L \subseteq L \mbox{ and } 
	L X \subseteq L \}.$
	\end{itemize}
	 
\end{definition}

\begin{remark}
If $\Lambda$ is an order in $K^{d\times d}$, then	$\Jac (\Lambda ) $ is a  two-sided ideal of $\Lambda $ that contains $\pi \Lambda $. The quotient $\Lambda /\Jac(\Lambda )$ is a semisimple 
	$\Ocal_K/\mathfrak{m}_K$-algebra and, for some $n$, one has
	$\Jac(\Lambda )^n \subseteq \pi \Lambda $. Moreover, $\Jac(\Lambda )$ is the unique pro-nilpotent ideal with semisimple quotient ring. For this and more, see for instance \cite[Chapter~1, Section~6]{Reiner/75}.
\end{remark}

\begin{definition}
Let $\Lambda$ be an order in $K^{d\times d}$. The {\em radical idealizer chain} $(\Omega_i)_{i\geq 0}$ of $\Lambda $ is recursively defined by
\[
\Omega _0 := \Lambda \textup{ and } \Omega_{i+1}=\Id(\Jac(\Omega_i)).
\]
\end{definition}

\begin{remark}\label{piJac}
The radical idealizer chain of an order $\Lambda $
is an ascending finite chain 
$\Omega _0 \subset \Omega_1 \subset 
\ldots 
\subset \Omega_{s}(= \Omega_{s+1}=\ldots )$; cf.\ \cite[Remark~3.8]{NebeSteel}.
Moreover, as $\pi \Lambda \subseteq \Jac(\Lambda )$, we have 
$$\Lambda \subseteq \Omega_1=\Id(\Jac(\Lambda )) \subset \frac{1}{\pi } \Lambda .$$ 
This yields an efficient algorithm to compute the radical idealizer chain 
for orders based on solving linear equations in the residue field; cf.\ \cite{NebeSteel}.
The sets of invariant lattices 
${\mathcal L} _i := Q(\Omega _{i}) $ form a descending chain 
$${\mathcal L} _0 \supset {\mathcal L} _1 \supset 
\ldots  \supset 
{\mathcal L} _{s },$$
where the last element ${\mathcal L}_{s } = Q(\Omega _{s })$  
is known to be a simplex in the building ${\mathcal B}_d(K)$; cf.\ \cite[Theorem~(39.14)]{Reiner/75}.
The length $s \geq 0$ of the radical idealizer chain is 
called the {\em radical idealizer length} of the order $\Lambda $.
\end{remark}


\begin{lemma}\label{lem:distance}
Let $\Lambda $ be an order in $K^{d\times d}$ and put $\Omega_1:=\Id(\Jac(\Lambda))$. 
	Then $$Q(\Omega_1) \subseteq Q(\Lambda ) \subseteq \B_1(Q(\Omega_1 ) ).$$
In particular all lattices in $Q(\Lambda )$ have distance at most one from $Q(\Omega_1)$. 
\end{lemma}

\begin{proof} 
As $\Omega _1 \supseteq \Lambda $, 
we know that $Q(\Omega_1 ) \subseteq Q(\Lambda )$ and thus we get
	$Q(\Omega_1) = \{ [\Omega_1 L ] \mid [L] \in Q(\Lambda ) \}$. 
Moreover, by \Cref{piJac}, we have $\Lambda \subseteq \Omega_1 \subseteq \frac{1}{\pi } \Lambda $, 
	so $L \subseteq \Omega_1 L \subseteq \frac{1}{\pi } L $ 
	and hence $\dist([L],[\Omega_1 L]) \leq 1$, for all 
	$[L] \in Q(\Lambda )$.
\end{proof}

\begin{lemma}\label{prop:rad-ideal-graduated}
	Let $M\in {\mathcal P}_d(\ZZ)$.
Then $\Id(\Jac(\Lambda(M+J_d))) = \Lambda (M)$.
\end{lemma} 

\begin{proof}  
As  $\dim (Q_{M+J_d}) = d-1$, we know by  
 \cite[Example~23]{EHNSS/21}  that 
	the Jacobson radical of $\Lambda({M+J_d})$ is equal to $\pi \Lambda (M) $.
	This is a $2$-sided principal ideal in the order 
	 $\Lambda (M)$, so 
	$\Id (\pi \Lambda (M)) = \Lambda (M)$.
\end{proof}

\begin{example}\label{ex:radIdealizer}

\begin{figure}[h] 
    \centering
    \includegraphics[scale=0.25]{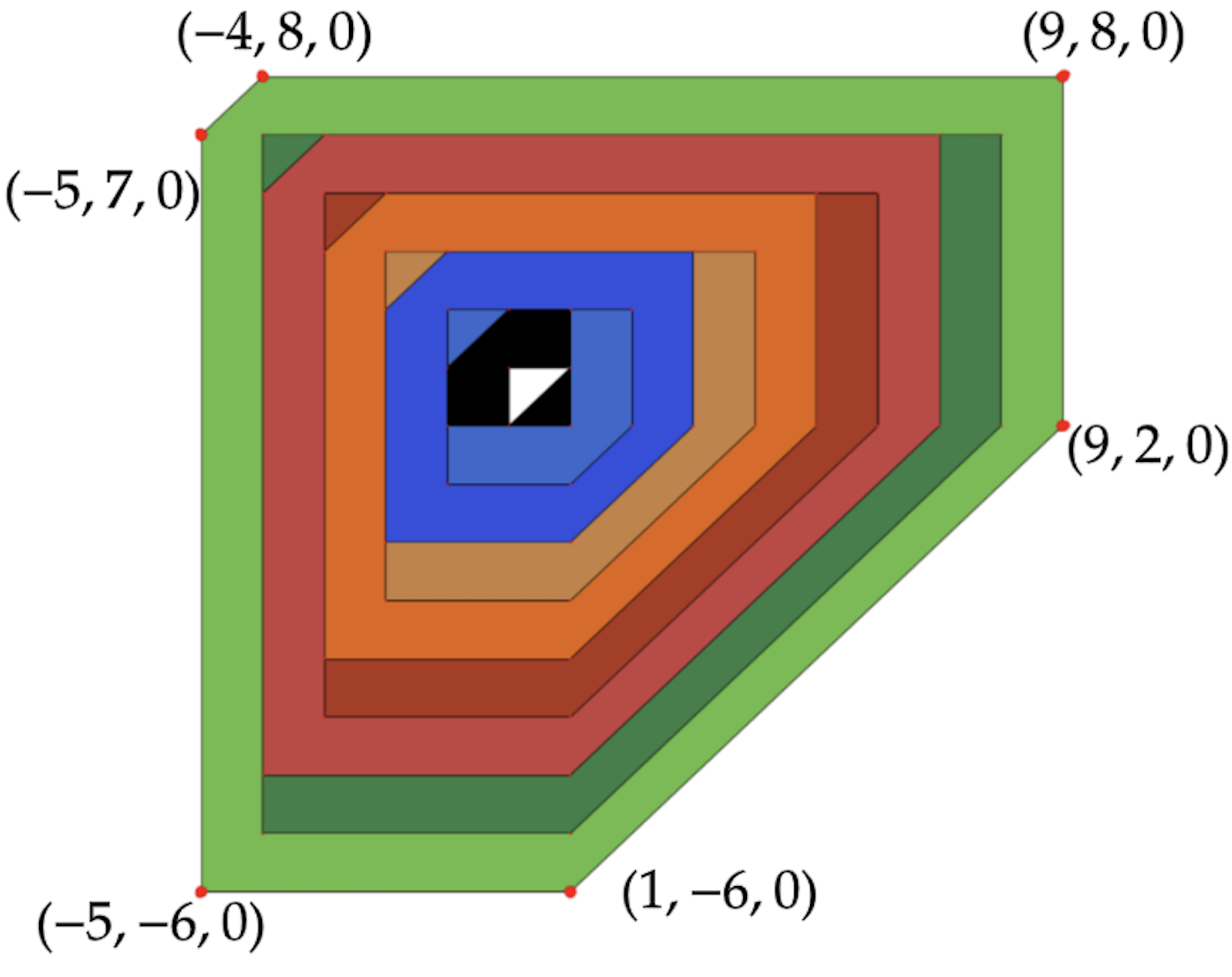}
    \caption{The radical idealizer process for the order $\Lambda(M)$ in \cref{ex:radIdealizer}}\label{fig:radical_idealizer}
\end{figure}

 Consider the configuration of lattice classes $[L_{u_1}], [L_{u_2}]$ and $[L_{u_3}]$ where
 \[
 u_1 = (0,12,5)\sim (-5,7,0), \quad u_2 = (7, 0, 6)\sim(1,-6,0), \quad \text{and } \quad u_3 = (9,8,0).
 \]
 In the notation of \cite{EHNSS/21}, this configuration corresponds to the matrix
 \[
 M = \begin{pmatrix} 0 & 7 & 9 \\ 12 & 0 & 8 \\ 5 & 6 & 0\end{pmatrix}
 \]
 and the decreasing sequence of polytropes $(Q(\Omega_i))_{i \geq 0}$ corresponding to the radical idealizer process for the order $\Lambda(M)$ is depicted in \cref{fig:radical_idealizer}. As expected, the last polytrope (in white) is indeed a simplex.

\end{example}

\section{Ball Orders} \label{ballorders} 

In this section, we define and study a first subfamily of the bolytrope orders, namely closed orders whose set of invariant lattices is a ball in ${\mathcal B}^0_d(K)$. 

\begin{definition}
	A \emph{ball order} in $K^{d\times d}$ is an order of the form $\Beta _r([L]) := \PZ(\B_r([L]))$, where 
	$L$ is a lattice in $K^d$ and $r$ is a non-negative integer.
\end{definition}

\begin{theorem}\label{prop:ball} 
Let $L$ be a lattice in $K^d$ and let $(e_1,\ldots ,e_d)$ be a basis of $L$. Let, moreover, $r$ be a non-negative integer. Then, with respect to $(e_1,\ldots,e_d)$, we have 
	$$\Beta _r([L]) = \{ X \in \Lambda (r J_d) \mid X_{11} \equiv \ldots \equiv X_{dd} \bmod \pi ^r  \} .$$
Moreover, $Q(\Beta_r([L])) = \B_r([L])$ and the ball $\B_r([L])$ is $\PZ $-closed.
\end{theorem}

\begin{proof}
Put $\Lambda=\{ X \in \Lambda (r J_d) \mid X_{11} \equiv \ldots \equiv X_{dd} \bmod \pi ^r  \}$  and $\Gamma = \End_{\Ocal _K}(L) =\Lambda (0) $. 
It follows from the definition of $\Lambda$ that $\pi ^r \Gamma \subseteq \Lambda $. If $L'$ is another lattice such that 
$\pi ^r L \subseteq L' \subseteq L $, then $\pi^r \Gamma L' \subseteq \pi^r \Gamma L = 
\pi^r L \subseteq L'$, which yields 
	$\pi ^r \Gamma \subseteq \Beta _r([L]) $. 
Now the lattice classes at distance at most $r$ from  $[L]$ can be described as submodules of $V_r=L/\pi^rL$. 
	In particular, the image $\overline{\Beta _r([L])}$ of $\Beta_r([L])$ in the endomorphism ring 
	$\End_{\Ocal_K}(V_r)\cong (\Ocal_K/\mathfrak{m}_K^r)^{d\times d}$ is equal to the collection
of all endomorphisms stabilizing every submodule of $V_r$. 
	This ensures that $$
	\overline{\Beta_r([L])}=(\Ocal_K/\mathfrak{m}_K^r)\Id_d = \overline{\Lambda }.$$
	As both orders $\Beta _r([L])$ and $\Lambda $ contain the kernel $\pi ^r \Gamma $
	of the projection $\Gamma\rightarrow\End_{\Ocal_K}(V_r)$, we conclude that $\Lambda = \Beta _r([L])= \PZ (\B_r([L]))$. We now show that $Q(\Beta_r([L]))=\B_r([L])$. To this end, let $[L'] \in Q(\Lambda )$. 
	Then 
	$[\Gamma L'] \in Q(\Gamma ) = \{ [ L ] \} $. Replacing $L'$ by some homothetic 
	lattice we hence may assume that $\Gamma L' = L $. 
	But $\pi ^r \Gamma \subseteq \Lambda \subseteq \End_{\Ocal _K}(L')$ so 
	$\pi^r \Gamma L' = \pi^r L \subseteq L' $ so $[L']\in \B_r([L])$.
\end{proof}

\begin{remark}\label{Ballradid} (Radical idealizer chain of ball orders)
     Let $r$ be a positive integer. Then the Jacobson radical of the ball order
	$\Beta _r([L]) = \PZ (\B_r([L])) $ is
	$\Jac(\Beta _r([L])) = \pi \Beta _{r-1}([L])$,  because
	$\pi \Beta _{r-1}([L])$ is a pro-nilpotent ideal of $\Beta _r([L])$ with
	simple quotient $\Beta _r([L])/ \pi \Beta _{r-1}([L])$ isomorphic to $ \Ocal _K/\mathfrak{m}_K$.
	Now $\pi \Beta _{r-1}([L])$ is a principal 2-sided ideal of
	$\Beta _{r-1}([L])$ so 
	$$\Id (\Jac(\Beta _r([L])))  = \Id (\pi \Beta _{r-1}([L])) = \Beta _{r-1}([L]) $$
        and the radical idealizer chain for ball orders is thus
	$$ \Beta _{r}([L]) \subset \Beta_{r-1}([L]) \subset \ldots \subset \Beta _1([L]) \subset
	\Beta _0([L]) = \End_{\Ocal_K} (L) .$$ The corresponding chain
        of $\PZ $-closed subsets of ${\mathcal B}_d^0(K)$ is
        $$\B_r([L]) \supset \B_{r-1} ([L])  \supset \ldots \supset \B_1([L])
	\supset \B_0([L]) = \{ [L] \} .$$
\end{remark}

The knowledge of the radical idealizer chain of ball orders allows to prove strong properties of ball orders, like the following. 

\begin{proposition}\label{prop:strong-prop-balls}
Let $r$ be a positive integer and $\Lambda $ a closed order in $K^{d\times d}$ such that $\Id(\Jac(\Lambda)) = \Beta _{r-1}([L])$. Then one has 
$\Beta _{r} ([L])\subseteq \Lambda \subseteq \Beta _{r-1}([L])$.
\end{proposition}

\begin{proof}
It follows from the hypotheses and the combination of \Cref{lem:distance} with \Cref{prop:ball} that 
$\B_{r-1}([L])\subseteq Q(\Lambda)\subseteq \B_r([L])$. The orders being closed, \Cref{Ballradid} yields that $\Beta _{r} ([L])\subseteq \Lambda \subseteq \Beta _{r-1}([L])$.
\end{proof}


\begin{definition}\label{def:star-config}
Let $r$ be a non-negative integer and $L$ a lattice in $K^d$.
	A {\em star configuration} $\star_r ([L]) $ with center $[L]$ and
	radius $r$ is a set 
	$$\star _r ([L]) = \{ [L_1],\ldots , [L_d], [L_{d+1}] \} $$ 
	such that the following hold:
	\begin{enumerate}[label=$(\arabic*)$]
	    \item $\pi ^r L \subseteq L_1,\ldots , L_{d+1} \subseteq L$,
	    \item for each $i\in\{1,\ldots,d+1\}$, one has $L_i / \pi^r L \cong \Ocal _K/\mathfrak{m}_K^r$,
	    \item for each $i\in\{1,\ldots,d+1\}$, one has $L= \sum _{j\neq i} L_j $. 
	\end{enumerate}
\end{definition}

When $r=1$, i.e.\ when $\Ocal_K/\mathfrak{m}_K^r$ is a field, the $1$-dimensional free $\Ocal _K/\mathfrak{m}_K^r$-modules $L_i/\pi ^r L$ of
$L/\pi^r L$ form a projective basis. In this sense, \Cref{def:star-config} generalizes the definition of a projective basis to modules over rings.

\begin{theorem}\label{prop:balls}
Let $r$ be a non-negative integer and let $L$ be a lattice in $K^d$. Let, moreover, $\star_r([L])$ denote a star configuration 
	with center $[L]$ and radius $r$. 
Then one has 
\[\Beta _r([L]) = \PZ (\star _r([L]) ) \
\textup{ and } \
\deg(\Beta_r([L]))\leq d.\]
\end{theorem} 

\begin{proof}
Write $\Lambda := \PZ (\star_r([L]) )$ and $ \star _r([L]) =: \{ [L_1],\ldots , [L_{d+1}] \} $. Since $\star_r([L])$ has radius $r$, we have that $\star_r([L]) \subseteq  
\B_r([L])$, so $\Lambda \supseteq \Beta _r([L])$.
We now claim that $\Lambda $ stabilizes all lattices 
	$L'$ with $\pi^r L \subseteq L' \subseteq L$.
To this end, write $\overline{L}=L/\pi^rL$ and use the bar notation for the submodules of $\overline{L}$.
For $1\leq i \leq d$ 
	let $e_i \in L_i$ be such that $\overline{\Ocal_Ke_i} = \overline{L_i} $. 
	Since $L_1+\ldots + L_d = L$, the set $\{ \overline{e_1},\ldots , \overline{e_d} \} $ is 
	a basis of the free module $\overline{L}$. 
	So there are $a_i \in \Ocal _K$ such that 
	$\overline{\Ocal_K \sum _{i=1}^d a_i e_i } = \overline{L_{d+1}} $. 
	Since $\star_r([L])$ is a star configuration,
	all $a_i$'s are units, so, replacing $e_i$ by $a_ie_i$, we assume, without loss of generality, that
	$L_{d+1} = \Ocal_K( e_1+\ldots + e_d ) + \pi^r L$.
 Since each $L_i$ is $\Lambda$-stable, 
the image of $\Lambda$ in 
$\End(\overline{L})\cong (\Ocal_K/\mathfrak{m}_K^r)^{d\times d}$ consists of 
	scalar matrices and so all submodules of $\overline{L}$ are stable. This yields the claim and so $\Beta _r([L]) = \PZ (\star _r([L]) )$. The order $\Beta _r([L])$ has degree at most $d$, because a star configuration has cardinality $d+1$.
\end{proof}

The following remark shows that ball orders in $K^{d\times d}$ can have degree smaller than $d$. 

\begin{remark}\label{ex:yassine}
   The degree of  $\Beta_r([\Ocal_K^4])$ is at most $3$, because $\Beta_r([\Ocal_K^4])$ is equal to the Plesken-Zassenhaus order of the following lattices (where the columns of the matrices are the basis elements): 
    \[
    \begin{pmatrix}
        1 \ & 0 \ & 0 & 0 \\
        0 \ & 1 \ & 0 & 0 \\
        1 \ & 0 \ & \pi^r & 0\\
        0 \ & 0 \ & 0 & \pi^r
    \end{pmatrix},  
    \begin{pmatrix}
        \pi^r & 0     & 0 \ & 0 \\
        0     & \pi^r & 0 \ & 0 \\
        0     & 0     & 1 \ & 0\\
        0     & 0     & 0 \ & 1
    \end{pmatrix},  
    \begin{pmatrix}
        1 \ & 0 & 0 & 0 \\
        0 \ & 1 & 0 & 0 \\
        1 \ & 1 & \pi^r & 0\\
        1 \ & 0 & 0 & \pi^r
    \end{pmatrix}, \text{ and }  
    \begin{pmatrix}
        \pi^r & 0 \  & 0     & 0 \\
            0 & 1 \  & 0     & 0 \\
            0 & 0 \  & \pi^r & 0\\
            0 & 1 \  & 0     & \pi^r
    \end{pmatrix}.
    \]
Via change of coordinates, one obtains that any ball order in $K^{4\times 4}$ has degree at most $3$.
\end{remark}

\section{Bolytrope Orders} \label{bolytroporders} 

Let $M\in {\mathcal P}_d(\ZZ)$. 
Recall, from Definition \ref{def:ball}, that the bolytrope 
$\B_r(M) $ is defined to be $\B_r(Q(\Lambda(M)))$. 

\begin{definition}
A \emph{bolytrope order} is an order of the form $\Beta _r(M):=\PZ(\B_r(M))$, where $M$ is an element of $\mathcal{P}_d(\ZZ )$ and $r$ is a non-negative integer.
\end{definition}

Until the end of the present section, fix $M\in\mathcal{P}_d(\ZZ)$ and an apartment $\mathcal{A}$ containing $Q(\Lambda(M))$. Let, moreover, $r$ be a non-negative integer. Then, by Lemma \ref{ballapp}, we have that
	$\B_r(M) \cap {\mathcal A} = Q(\Lambda({M+rJ_d})) $, in particular
        $\Beta _r(M) \subseteq \Lambda (M+rJ_d) $.
Put 
	$$\Lambda _r(M) = \{ X \in \Lambda (M+r J_d) \mid X_{11} \equiv \ldots \equiv X_{dd} \bmod \pi ^r \} .$$
We will show that $\Lambda _r(M) = \Beta _r(M)$ and 
$Q(\Lambda _r(M) ) = \B_r(M)$ is $\PZ $-closed; cf.\ \Cref{th:bolytropes}.

\begin{lemma}\label{lem:bolytrope1} 
Let $[L]$ be a lattice class in $Q(\Lambda(M))$. Then 
	$\Lambda _r(M) = \Lambda (M+rJ_d) \cap \Beta _r([L])$ and $\Lambda_r(M)$ is a closed order.
\end{lemma}

\begin{proof}
Let $(e_1,\ldots ,e_d)$ be a basis of $L$ that is also a frame basis defining the apartment $\mathcal{A}$. Then, with respect to this basis, $[L]=[\Ocal_K^d]$ and thus
 $\Lambda (M) \subseteq \End _{\Ocal _K}(L) = \Ocal _{K}^{d\times d}=\Lambda(0)$. It follows in particular that $M$ has non-negative entries. The explicit description of the ball order in \Cref{prop:ball} allows to deduce that 
  $\Lambda_r(M) = \Lambda (M+rJ_d) \cap \Beta _r([L])$.
  Since $\Lambda(M+rJ_d)$ and $\Beta_r([L])$ are closed orders, then so is $\Lambda_r(M)$.
\end{proof} 

\begin{lemma}\label{lem:bolytrope2} 
One has	$\B_r(M) \subseteq Q(\Lambda_r(M) )$ and  $\Lambda_r(M) \subseteq \Beta _r(M)$. 
\end{lemma}

\begin{proof}
We first show that $\B_r(M) \subseteq Q(\Lambda_r(M) )$. For this, let $[L'] \in \B_r(M)$ and let $[L]\in Q(\Lambda (M))$ be such that $\dist([L'],[L]) \leq r$. 
Then the combination of \Cref{Ballradid} and \Cref{lem:bolytrope1} yields that 
\[
[L']\in \B_r([L])=Q(\Beta_r([L]))\subseteq Q(\Lambda_r(M)). 
\]
To conclude, the inclusion $\B_r(M) \subseteq Q(\Lambda_r(M) )$ implies that $\Lambda_r(M) \subseteq \Beta _r(M)$.
\end{proof} 

To prove that $\Beta _r(M) = \Lambda _r(M)$ we use the 
 radical idealizer chain of $\Lambda_r(M)$, which we describe in the following remark. 

\begin{remark} \label{bolyradid} 
Assume that $r\geq 1$. Then, similarly to what is done in \Cref{Ballradid}, one sees that 
	$\Jac(\Lambda _r(M)) = \pi \Lambda _{r-1}(M)$ is a 
	2-sided principal ideal of $\Lambda _{r-1}(M)$ and  hence
	$\Id (\Jac(\Lambda _r(M)))   = \Lambda _{r-1}(M) $.
\end{remark}

\begin{lemma}\label{lem:bolytrope3}
One has	$Q(\Lambda _r(M)) = \B_r(M)$. 
\end{lemma} 

\begin{proof}
	Lemma \ref{lem:bolytrope2} shows that 
	$\B_r(M) \subseteq Q(\Lambda_r(M) )$. 
	For the opposite inclusion, we rely on
	\Cref{bolyradid} to proceed by induction on $r$. Assume first that $r=0$. Then $Q(\Lambda _0(M)) = Q(\Lambda (M) )=\B_0(M)$ and so we are done. Now assume that $r>0$ and that $Q(\Lambda _{r-1}(M))= \B_{r-1}(M)$. The fact that $\Lambda _{r-1}(M) = \Id (\Jac(\Lambda _r(M))) $  together with
	Lemma \ref{lem:distance} then yields that 
	\[Q(\Lambda_r(M)) \subseteq \B_1(Q(\Lambda_{r-1}(M))) = 
	\B_1(\B_{r-1}(M)) \subseteq \B_r(M) .\]
	This concludes the proof.
\end{proof}

The following is the main result of this section and of the paper. 

\begin{theorem}\label{th:bolytropes}
The following hold: 
\[\Lambda _r(M) = \Beta _r(M)\ \textup{ and } \ Q(\Beta _r(M))) = \B_r(M).\] 
In particular bolytrope orders are closed and 
bolytropes are $\PZ $-closed. 
\end{theorem}

\begin{proof}
As a consequence of \Cref{lem:bolytrope1}, both $\Lambda_r(M)$ and $\Beta_r(M)$ are closed orders. We are now done thanks to \Cref{lem:bolytrope3}.
\end{proof}

\begin{corollary}\label{cor:radid-boly}
The beginning of the radical idealizer chain for bolytrope orders is
	$$ \Beta _{r}(M) \subset \Beta_{r-1}(M) \subset \ldots \subset \Beta _1(M) \subset
	\Beta _0(M) = \Lambda (M).$$ 
The first $r+1$ elements in the corresponding chain
        of $\PZ $-closed subsets of ${\mathcal B}_d^0(K)$ are
        $$\B_r(M) \supset \B_{r-1} (M)  \supset \ldots \supset \B_1(M)
        \supset Q(\Lambda(M)).$$
	\end{corollary}

Note that $\Lambda (M)$ is the first term in the radical idealizer
	process that is a graduated order. 
The polytrope $Q(\Lambda(M))$ is hence canonically determined by the bolytrope 
$\B_r(M)$ and called the {\em central polytrope} of $\B_r(M)$.

In analogy with ball orders, we 
obtain the following stronger property of bolytrope orders.

\begin{corollary}\label{bolymin}
Assume that $r\geq 1$ and let $\Lambda $ be a closed order in $K^{d\times d}$ such that $\Id(\Jac(\Lambda)) = \Beta _{r-1}(M)$.  Then one has
$\Beta _{r} (M)\subseteq \Lambda \subseteq \Beta _{r-1}(M)$.
\end{corollary}

\begin{proof}
Analogous to the proof of \Cref{prop:strong-prop-balls}.
\end{proof}

\begin{theorem}\label{thm:bolystar} 
	Let $[P_1],\ldots , [P_d]$ be the distinct classes of projective 
	$\Lambda (M+rJ_d)$-lattices. 
	Then there is a lattice class $[L_{d+1}] \in \B_r(M)$, such that 
	$$\Beta _r(M) = \PZ ([P_1],\ldots , [P_d], [L_{d+1}] ) .$$
Moreover, the degree of $\Beta_r(M)$ is at most $d$.
\end{theorem}

\begin{proof}
	As a consequence of \Cref{polytrope}, we have that $\Lambda(M+rJ_d)=\PZ([P_1],\ldots , [P_d])$. In particular, for any lattice class $[L_{d+1}]\in \B_r(M)$, \Cref{lem:bolytrope1,th:bolytropes} imply that
	$$\Beta_r(M) \subseteq 
	\PZ ( [P_1],\ldots , [P_d], [L_{d+1}] ) \subseteq \Lambda (M+rJ_d).$$
	To construct $L_{d+1}$ such that the inclusion $\Beta_r(M) \supseteq 
	\PZ ([P_1],\ldots , [P_d], [L_{d+1}] )$ holds, choose 
	$[L]\in Q(\Lambda(M))$ and a lattice basis $(e_1,\ldots ,e_d)$ of $L$ that is also a 
	frame basis for some apartment containing $Q(\Lambda({M+rJ_d}))$. 
Define $L_{d+1} := \Ocal_K( e_1+\ldots + e_d ) + \pi^r L$ and, 
	for each $i=1,\ldots , d$, put $L_i:= \Ocal_K e_i + \pi^r L \in Q(\Lambda({M+rJ_d}))$. 
	Then $\{[L_1],\ldots , [L_d],[L_{d+1}]\}$ 
	is a star configuration with center $[L]$ and radius $r$. 
	By Theorem \ref{prop:balls}, we thus have 
	$$\PZ( [L_1],\ldots , [L_d],[L_{d+1}] ) = \Beta_r([L]),$$
	which, together with \Cref{lem:bolytrope1} and \Cref{th:bolytropes}, implies that $\PZ ([P_1],\ldots , [P_d], [L_{d+1}] )$ is contained in $ \Lambda (M+rJ_d)\cap \Beta _r([L]) = \Lambda _r(M) = \Beta _r(M)  $. 
\end{proof}

\section{When the building is a tree} \label{d=2}

Throughout this section, assume that $d=2$. Then the  building ${\mathcal B}_2(K)$ is an infinite tree. 
Apartments correspond to infinite paths in the tree and 
the bounded convex subsets of ${\mathcal B}_2(K)$ are the bounded
subtrees. For more on this and other trees, see for instance \cite{SerreTrees}.

The following is the main result of this section, which extends
\cite[Theorem 2]{Tu/11} beyond the case of finite residue fields.

\begin{theorem}  \label{thm:d=2}
	Let $\Lambda$ be a closed order in $K^{2\times 2}$. 
	Then there are $r,m\in \ZZ _{\geq 0}$ such that 
	$$ \Lambda = \Beta _r \left(\begin{pmatrix}
	0 & m \\ 0 & 0
	\end{pmatrix} \right)  =\{ X \in \Ocal_K^{2\times 2} \mid X_{12}\in \mathfrak{m}_K^{m+r}, 
		X_{21} \in \mathfrak{m}_K^r, 
		X_{11} \equiv X_{22} \bmod\pi ^r  \} .$$
\end{theorem} 

\begin{proof} 
Put 
$R:=
\max \{ \dist ([L],[L']) \mid [L],[L']\in Q(\Lambda ) \}$
and	let $[L_1],[L_2]\in Q(\Lambda )$ be such that 
	$R = \dist ([L_1],[L_2])$. 
	Then the convex hull 
	$${\mathcal L} = Q(\PZ([L_1],[L_2])) \subseteq Q(\Lambda ) $$
	is a line segment and is hence contained in an apartment ${\mathcal A}$.
Define $$r:=\max\{\dist([L], \mathcal{L}) \lvert  [L]\in Q(\Lambda)\}$$ and
	let $[L_3]\in Q(\Lambda)$ be such that $r=\dist([L_3],{\mathcal L})$.
	Let, moreover, $[L_0] \in {\mathcal L}$ denote the unique 
	lattice class in $\mathcal{L}$ satisfying $\dist([L_3],[L_0]) = r$.

	
	\begin{figure}[ht]
	\begin{center}
		\begin{tikzpicture}[scale=0.5]
		  \draw[fill=black] (-6,0,0) circle (3pt) node[anchor=north]{\footnotesize{$[L_1]$}};	  
		 \draw[fill=black]  (0,0,0) circle (3pt) node[anchor=north]{\footnotesize{$[L_1']$}};
		  \draw[fill=black] (-6,0,0) --node[anchor=south]{} (3,0,0);
		  \draw[fill=black] (3,0,0) circle (3pt) node[anchor=north]{\footnotesize{$[L_0]$}};
		  \draw[fill=black] (3,0,0);
		  \draw[fill=black] (3,0,0) -- node[anchor=west]{\footnotesize{$r$}} (3,3,0) circle (3pt) node[anchor=south]{\footnotesize{$[L_3]$}};	  
		  \draw[fill=black] (3,0,0) --(6,0,0) circle (3pt) node[anchor=north]{\footnotesize{$[L_2']$}};
		  \draw[fill=black] (6,0,0) --(8,0,0) circle (3pt) node[anchor=north]{\footnotesize{$[L_2]$}};
          \draw[color=blue!50,very thick ](3,0,0) circle (3);		    
		\end{tikzpicture}
	\end{center}
    \end{figure} 
	
	Now choose a frame basis $(e_1,e_2)$ for ${\mathcal A}$ such that, with respect to this basis, there exists an integer $m$ such that 
    $[L_1]= [L_{(0,r)}]$ and 
    $[L_2]=[L_{(m+r,0)}]$.
    It follows from the definition of $R$ that $$R + 1 = \lvert {\mathcal L} \rvert  = m + 2r + 1.$$
    With respect to the chosen basis, note now that ${\mathcal L} = Q(\Lambda({M+rJ_2}))$ and hence 
    \[\Lambda \subseteq \Lambda(M+rJ_2).\]
    Moreover, if $[L'_1] $ and $[L'_2] \in {\mathcal L} $ 
		are the two lattice classes at distance $r$ from $[L_0]$ and such that $\dist([L_1'],[L_2'])=2r$, then
		the set $\{[L_3],[L'_1],[L'_2]\} $ is a star configuration
			with radius $r$ and center $[L_0]$.
		As a consequence of the definition of $\mathcal{L}$, such lattice classes $[L_1'],[L_2']$ exist and thus Theorem \ref{prop:balls} ensures that \[\Lambda \subseteq \Beta _r([L_0]).\]
	We have proven that  $\Lambda \subseteq \Beta _r([L_0]) \cap \Lambda(M+rJ_2)$ and so $\Lambda \subseteq \Beta_r(M) $, thanks to \Cref{lem:bolytrope1}. 
	As $Q(\Lambda ) \subseteq \B_r(M) = Q(\Beta _r(M))$,
	we obtain 
	$\Lambda = \Beta_r(M) $ as stated in the theorem.
\end{proof}

\begin{figure}[ht]
\begin{center}
    \tikzset{every picture/.style={line width = 0.4pt}}          
    
    \begin{tikzpicture}[x=0.70pt,y=0.70pt,yscale=-1,xscale=1]
    
     \draw  [fill={rgb, 255:red, 65; green, 117; blue, 5 }  ,fill opacity=1 ] (268.81,195.84) .. controls (268.84,193.52) and (270.62,191.67) .. (272.78,191.7) .. controls (274.94,191.73) and (276.67,193.64) .. (276.65,195.96) .. controls (276.62,198.27) and (274.84,200.13) .. (272.68,200.1) .. controls (270.51,200.07) and (268.78,198.16) .. (268.81,195.84) -- cycle ;
     
     \draw  [fill={rgb, 255:red, 65; green, 117; blue, 5 }  ,fill opacity=1 ] (299.03,174.06) .. controls (299.05,171.74) and (300.83,169.88) .. (303,169.91) .. controls (305.16,169.95) and (306.89,171.85) .. (306.86,174.17) .. controls (306.83,176.49) and (305.05,178.34) .. (302.89,178.31) .. controls (300.73,178.28) and (299,176.38) .. (299.03,174.06) -- cycle ;
     
     \draw  [fill={rgb, 255:red, 65; green, 117; blue, 5 }  ,fill opacity=1 ] (243.01,173.33) .. controls (243.04,171.01) and (244.81,169.15) .. (246.98,169.19) .. controls (249.14,169.22) and (250.87,171.12) .. (250.84,173.44) .. controls (250.81,175.76) and (249.04,177.62) .. (246.87,177.58) .. controls (244.71,177.55) and (242.98,175.65) .. (243.01,173.33) -- cycle ;
     
     \draw  [fill={rgb, 255:red, 65; green, 117; blue, 5 }  ,fill opacity=1 ] (210.57,183.42) .. controls (210.6,181.2) and (212.3,179.43) .. (214.36,179.46) .. controls (216.43,179.49) and (218.08,181.31) .. (218.05,183.52) .. controls (218.03,185.74) and (216.33,187.51) .. (214.26,187.48) .. controls (212.2,187.45) and (210.54,185.63) .. (210.57,183.42) -- cycle ;
 
     \draw   (252.02,261.95) .. controls (252.05,259.64) and (253.82,257.78) .. (255.99,257.81) .. controls (258.15,257.84) and (259.88,259.75) .. (259.85,262.07) .. controls (259.82,264.39) and (258.05,266.24) .. (255.88,266.21) .. controls (253.72,266.18) and (251.99,264.27) .. (252.02,261.95) -- cycle ;
    
     \draw   (286.1,261.53) .. controls (286.13,259.22) and (287.9,257.36) .. (290.07,257.39) .. controls (292.23,257.42) and (293.96,259.33) .. (293.93,261.65) .. controls (293.9,263.97) and (292.13,265.82) .. (289.96,265.79) .. controls (287.8,265.76) and (286.07,263.85) .. (286.1,261.53) -- cycle ;

    \draw [color={rgb, 255:red, 65; green, 117; blue, 5 }  ,draw opacity=1 ][line width=1.5]    (250.01,176.52) -- (269.59,193.32) ;

    \draw [color={rgb, 255:red, 65; green, 117; blue, 5 }  ,draw opacity=1 ][line width=1.5]    (276.14,193.21) -- (299.81,176.58) ;

    \draw [color={rgb, 255:red, 254; green, 19; blue, 19 }  ,draw opacity=1 ][line width=1.5]    (272.68,200.1) -- (272.9,225.53) ;

    \draw [color={rgb, 255:red, 65; green, 117; blue, 5 }  ,draw opacity=1 ][line width=1.5]    (218.05,183.52) -- (243.4,175.85) ;

    \draw    (258.01,258.29) -- (270.43,232.86) ;

    \draw [color={rgb, 255:red, 255; green, 0; blue, 0 }  ,draw opacity=1 ][line width=1.5]    (245.02,170.03) -- (227.06,146.98) ;

    \draw    (287.78,257.87) -- (275.52,232.44) ;

    \draw [color={rgb, 255:red, 255; green, 0; blue, 0 }  ,draw opacity=1 ][line width=1.5]    (303,169.91) -- (310.68,142.63) ;

    \draw   (332.99,112.71) .. controls (331.7,114.59) and (329.22,115.06) .. (327.45,113.76) .. controls (325.68,112.45) and (325.29,109.86) .. (326.58,107.98) .. controls (327.87,106.09) and (330.35,105.63) .. (332.12,106.93) .. controls (333.89,108.24) and (334.28,110.83) .. (332.99,112.71) -- cycle ;

    \draw   (364.57,96.26) .. controls (363.28,98.15) and (360.8,98.62) .. (359.03,97.31) .. controls (357.26,96) and (356.87,93.42) .. (358.16,91.53) .. controls (359.45,89.65) and (361.93,89.18) .. (363.7,90.49) .. controls (365.47,91.79) and (365.86,94.38) .. (364.57,96.26) -- cycle ;

    \draw   (336.2,76.43) .. controls (334.91,78.31) and (332.43,78.78) .. (330.66,77.47) .. controls (328.89,76.16) and (328.5,73.58) .. (329.79,71.69) .. controls (331.08,69.81) and (333.56,69.34) .. (335.33,70.65) .. controls (337.1,71.95) and (337.49,74.54) .. (336.2,76.43) -- cycle ;

    \draw    (313.76,134.72) -- (327.45,113.76) ;

    \draw    (357.62,95.72) -- (333.49,109.2) ;

    \draw    (332.8,78.43) -- (329.06,106.53) ;

    \draw   (233.55,111.15) .. controls (233.21,113.44) and (231.2,115.02) .. (229.06,114.68) .. controls (226.92,114.34) and (225.47,112.2) .. (225.81,109.91) .. controls (226.15,107.61) and (228.16,106.03) .. (230.3,106.38) .. controls (232.44,106.72) and (233.89,108.86) .. (233.55,111.15) -- cycle ;

    \draw   (254.65,81.62) .. controls (254.31,83.91) and (252.3,85.49) .. (250.16,85.15) .. controls (248.02,84.81) and (246.56,82.67) .. (246.91,80.38) .. controls (247.25,78.09) and (249.26,76.51) .. (251.39,76.85) .. controls (253.53,77.19) and (254.99,79.33) .. (254.65,81.62) -- cycle ;

    \draw   (220.84,77.12) .. controls (220.5,79.41) and (218.49,80.99) .. (216.36,80.65) .. controls (214.22,80.3) and (212.76,78.17) .. (213.1,75.87) .. controls (213.44,73.58) and (215.45,72) .. (217.59,72.34) .. controls (219.73,72.69) and (221.18,74.82) .. (220.84,77.12) -- cycle ;

    \draw    (225.85,139.89) -- (229.06,114.68) ;

    \draw    (248.22,84.39) -- (232.49,107.77) ;

    \draw    (218.68,80.5) -- (227.4,107.45) ;

    \draw   (199.13,123.72) .. controls (201,124.99) and (201.55,127.54) .. (200.36,129.41) .. controls (199.16,131.28) and (196.68,131.76) .. (194.81,130.48) .. controls (192.94,129.2) and (192.39,126.65) .. (193.58,124.79) .. controls (194.77,122.92) and (197.26,122.44) .. (199.13,123.72) -- cycle ;

    \draw   (163.94,121.48) .. controls (165.81,122.76) and (166.36,125.31) .. (165.17,127.18) .. controls (163.98,129.04) and (161.5,129.52) .. (159.63,128.25) .. controls (157.75,126.97) and (157.2,124.42) .. (158.4,122.55) .. controls (159.59,120.69) and (162.07,120.21) .. (163.94,121.48) -- cycle ;

    \draw    (221.15,142.94) -- (200.36,129.41) ;

    \draw    (181.74,98.76) -- (195.72,123.3) ;

    \draw    (166.02,124.93) -- (193.31,127.96) ;

    \draw   (185.12,94.09) .. controls (185.4,96.39) and (183.89,98.48) .. (181.74,98.76) .. controls (179.6,99.03) and (177.63,97.39) .. (177.35,95.09) .. controls (177.07,92.79) and (178.58,90.7) .. (180.73,90.42) .. controls (182.88,90.15) and (184.84,91.79) .. (185.12,94.09) -- cycle ;

    \draw  [fill={rgb, 255:red, 0; green, 0; blue, 255 }  ,fill opacity=1 ] (183.88,171.14) .. controls (186.04,171.76) and (187.32,174.01) .. (186.74,176.17) .. controls (186.16,178.32) and (183.93,179.57) .. (181.78,178.95) .. controls (179.62,178.33) and (178.34,176.08) .. (178.92,173.92) .. controls (179.5,171.76) and (181.72,170.52) .. (183.88,171.14) -- cycle ;

    \draw   (158.04,145.97) .. controls (160.2,146.59) and (161.48,148.84) .. (160.9,151) .. controls (160.32,153.16) and (158.09,154.4) .. (155.94,153.78) .. controls (153.78,153.16) and (152.5,150.91) .. (153.08,148.75) .. controls (153.66,146.6) and (155.88,145.35) .. (158.04,145.97) -- cycle ;

    \draw   (149.73,180.19) .. controls (151.89,180.81) and (153.17,183.06) .. (152.59,185.22) .. controls (152.01,187.37) and (149.79,188.62) .. (147.63,188) .. controls (145.47,187.38) and (144.19,185.13) .. (144.77,182.97) .. controls (145.35,180.81) and (147.57,179.57) .. (149.73,180.19) -- cycle ;
    \draw [color={rgb, 255:red, 255; green, 0; blue, 0 }  ,draw opacity=1 ][line width=1.5]    (210.54,182.44) -- (186.74,176.17) ;
    \draw    (159.94,152.91) -- (180.51,171.82) ;
    \draw    (152.73,182.81) -- (179.61,177.03) ;
    \draw  [fill={rgb, 255:red, 65; green, 117; blue, 5 }  ,fill opacity=1 ] (189.34,204.1) .. controls (191.01,202.55) and (193.53,202.65) .. (194.98,204.31) .. controls (196.42,205.98) and (196.24,208.59) .. (194.58,210.14) .. controls (192.91,211.69) and (190.39,211.6) .. (188.95,209.93) .. controls (187.5,208.27) and (187.68,205.66) .. (189.34,204.1) -- cycle ;
    \draw [color={rgb, 255:red, 65; green, 117; blue, 5 }  ,draw opacity=1 ][line width=1.5]    (212.87,186.89) -- (194.98,204.31) ;
    \draw [color={rgb, 255:red, 255; green, 0; blue, 0 }  ,draw opacity=1 ][line width=1.5]    (161.63,215.23) -- (188.1,207.41) ;
    \draw [color={rgb, 255:red, 255; green, 0; blue, 0 }  ,draw opacity=1 ][line width=1.5]    (182.13,237.58) -- (191.85,211) ;
    \draw   (303.33,106.18) .. controls (304.15,108.35) and (303.18,110.75) .. (301.16,111.55) .. controls (299.15,112.34) and (296.85,111.23) .. (296.03,109.06) .. controls (295.21,106.9) and (296.18,104.5) .. (298.2,103.7) .. controls (300.21,102.91) and (302.51,104.02) .. (303.33,106.18) -- cycle ;
    \draw   (307.24,69.74) .. controls (308.06,71.91) and (307.09,74.31) .. (305.07,75.1) .. controls (303.06,75.9) and (300.76,74.79) .. (299.94,72.62) .. controls (299.12,70.45) and (300.09,68.05) .. (302.11,67.26) .. controls (304.13,66.46) and (306.42,67.57) .. (307.24,69.74) -- cycle ;
    \draw   (276.22,82.53) .. controls (277.01,84.71) and (276.01,87.1) .. (273.99,87.87) .. controls (271.96,88.64) and (269.68,87.5) .. (268.88,85.32) .. controls (268.09,83.15) and (269.09,80.76) .. (271.12,79.99) .. controls (273.14,79.22) and (275.43,80.36) .. (276.22,82.53) -- cycle ;

    \draw    (310.63,135.09) -- (301.16,111.55) ;
    \draw    (303.02,75.43) -- (300.78,103.79) ;
    \draw    (275.58,87.16) -- (296.21,106.12) ;
    \draw  [fill={rgb, 255:red, 65; green, 117; blue, 5 }  ,fill opacity=1 ] (334.81,179.42) .. controls (332.57,179.36) and (330.81,177.49) .. (330.87,175.25) .. controls (330.93,173.02) and (332.79,171.25) .. (335.03,171.32) .. controls (337.27,171.38) and (339.04,173.24) .. (338.98,175.48) .. controls (338.92,177.72) and (337.05,179.48) .. (334.81,179.42) -- cycle ;
    \draw  [fill={rgb, 255:red, 65; green, 117; blue, 5 }  ,fill opacity=1 ] (365.21,162.11) .. controls (362.97,162.04) and (361.2,160.18) .. (361.26,157.94) .. controls (361.32,155.7) and (363.19,153.94) .. (365.43,154) .. controls (367.67,154.06) and (369.44,155.93) .. (369.38,158.17) .. controls (369.32,160.4) and (367.45,162.17) .. (365.21,162.11) -- cycle ;
    \draw [color={rgb, 255:red, 65; green, 117; blue, 5 }  ,draw opacity=1 ][line width=1.5]    (306.3,175.12) -- (330.87,175.25) ;
    \draw [color={rgb, 255:red, 255; green, 0; blue, 0 }  ,draw opacity=1 ][line width=1.5]    (362.3,191.13) -- (337.92,177.91) ;
    \draw [color={rgb, 255:red, 65; green, 117; blue, 5 }  ,draw opacity=1 ][line width=1.5]    (362.32,160.31) -- (337.59,172.64) ;
    \draw   (426.14,188.69) .. controls (423.91,188.4) and (422.33,186.36) .. (422.6,184.14) .. controls (422.88,181.92) and (424.9,180.35) .. (427.12,180.64) .. controls (429.34,180.93) and (430.93,182.97) .. (430.65,185.19) .. controls (430.38,187.41) and (428.36,188.98) .. (426.14,188.69) -- cycle ;
    \draw   (429.75,155.99) .. controls (427.53,155.7) and (425.95,153.66) .. (426.22,151.44) .. controls (426.49,149.22) and (428.51,147.65) .. (430.74,147.94) .. controls (432.96,148.23) and (434.54,150.27) .. (434.27,152.49) .. controls (434,154.71) and (431.97,156.28) .. (429.75,155.99) -- cycle ;
    \draw [color={rgb, 255:red, 255; green, 0; blue, 0 }  ,draw opacity=1 ][line width=1.5]    (369.07,160.51) -- (393.5,163.13) ;
    \draw    (423.29,182.11) -- (400.27,166.49) ;
    \draw    (426.22,151.44) -- (400.44,161.21) ;
    \draw  [fill={rgb, 255:red, 65; green, 117; blue, 5 }  ,fill opacity=1 ] (392.13,138.45) .. controls (390.38,139.9) and (387.87,139.66) .. (386.52,137.91) .. controls (385.17,136.16) and (385.5,133.56) .. (387.25,132.11) .. controls (389,130.67) and (391.51,130.91) .. (392.86,132.66) .. controls (394.21,134.41) and (393.88,137.01) .. (392.13,138.45) -- cycle ;
    \draw [color={rgb, 255:red, 65; green, 117; blue, 5 }  ,draw opacity=1 ][line width=1.5]    (367.67,154.21) -- (386.52,137.91) ;
    \draw [color={rgb, 255:red, 255; green, 9; blue, 0 }  ,draw opacity=1 ][line width=1.5]    (420.43,129.04) -- (393.57,135.23) ;
    \draw [color={rgb, 255:red, 255; green, 0; blue, 0 }  ,draw opacity=1 ][line width=1.5]    (401.23,105.48) -- (390.02,131.42) ;
    \draw   (394.13,209.78) .. controls (392.05,208.9) and (391.04,206.51) .. (391.86,204.44) .. controls (392.68,202.37) and (395.02,201.4) .. (397.1,202.28) .. controls (399.17,203.15) and (400.19,205.54) .. (399.37,207.62) .. controls (398.55,209.69) and (396.2,210.66) .. (394.13,209.78) -- cycle ;
    \draw   (416.96,237.89) .. controls (414.88,237.02) and (413.87,234.62) .. (414.69,232.55) .. controls (415.51,230.48) and (417.85,229.51) .. (419.93,230.39) .. controls (422,231.26) and (423.02,233.66) .. (422.2,235.73) .. controls (421.38,237.8) and (419.03,238.77) .. (416.96,237.89) -- cycle ;
    \draw   (428.65,207.28) .. controls (426.58,206.41) and (425.56,204.02) .. (426.38,201.95) .. controls (427.2,199.87) and (429.55,198.9) .. (431.62,199.78) .. controls (433.7,200.66) and (434.71,203.05) .. (433.89,205.12) .. controls (433.07,207.19) and (430.73,208.16) .. (428.65,207.28) -- cycle ;
    \draw    (368.92,195.34) -- (391.86,204.44) ;
    \draw    (415.86,230.78) -- (397.55,209.5) ;
    \draw    (426.38,201.95) -- (399.04,204.45) ;
    \draw   (373.71,225.92) .. controls (372.83,223.78) and (373.73,221.34) .. (375.73,220.47) .. controls (377.72,219.6) and (380.04,220.62) .. (380.92,222.75) .. controls (381.8,224.89) and (380.89,227.33) .. (378.9,228.2) .. controls (376.91,229.07) and (374.58,228.05) .. (373.71,225.92) -- cycle ;
    \draw   (370.79,262.52) .. controls (369.91,260.38) and (370.82,257.94) .. (372.81,257.07) .. controls (374.8,256.2) and (377.13,257.22) .. (378,259.35) .. controls (378.88,261.49) and (377.97,263.93) .. (375.98,264.8) .. controls (373.99,265.68) and (371.66,264.65) .. (370.79,262.52) -- cycle ;
    \draw   (399.91,249.31) .. controls (399.04,247.18) and (399.94,244.74) .. (401.94,243.87) .. controls (403.93,242.99) and (406.25,244.01) .. (407.13,246.15) .. controls (408,248.28) and (407.1,250.72) .. (405.11,251.6) .. controls (403.12,252.47) and (400.79,251.45) .. (399.91,249.31) -- cycle ;
    \draw    (365.63,197.29) -- (375.73,220.47) ;
    \draw    (374.85,256.66) -- (376.32,228.21) ;
    \draw    (401.93,243.87) -- (380.82,225.71) ;
    \draw   (226.47,277.06) .. controls (228.36,275.82) and (230.83,276.35) .. (231.98,278.25) .. controls (233.14,280.14) and (232.53,282.68) .. (230.64,283.92) .. controls (228.74,285.16) and (226.27,284.62) .. (225.12,282.73) .. controls (223.97,280.83) and (224.57,278.29) .. (226.47,277.06) -- cycle ;
    \draw   (191.09,279.93) .. controls (192.99,278.7) and (195.46,279.23) .. (196.61,281.13) .. controls (197.77,283.02) and (197.16,285.56) .. (195.27,286.8) .. controls (193.37,288.04) and (190.9,287.5) .. (189.75,285.61) .. controls (188.6,283.71) and (189.2,281.17) .. (191.09,279.93) -- cycle ;
    \draw   (209.98,309.31) .. controls (211.87,308.07) and (214.34,308.6) .. (215.5,310.5) .. controls (216.65,312.39) and (216.04,314.93) .. (214.15,316.17) .. controls (212.25,317.41) and (209.78,316.87) .. (208.63,314.98) .. controls (207.48,313.08) and (208.08,310.54) .. (209.98,309.31) -- cycle ;
    \draw    (252.47,264.22) -- (231.98,278.25) ;
    \draw    (197.33,283.14) -- (224.69,280.1) ;
    \draw    (213.87,308.77) -- (227.81,284.29) ;
    \draw   (253.71,295.04) .. controls (253.72,292.72) and (255.48,290.85) .. (257.65,290.86) .. controls (259.81,290.87) and (261.56,292.76) .. (261.55,295.08) .. controls (261.54,297.4) and (259.78,299.27) .. (257.61,299.26) .. controls (255.45,299.25) and (253.7,297.36) .. (253.71,295.04) -- cycle ;
    \draw   (237.07,327.47) .. controls (237.08,325.15) and (238.84,323.28) .. (241,323.29) .. controls (243.17,323.3) and (244.91,325.19) .. (244.9,327.51) .. controls (244.9,329.83) and (243.13,331.71) .. (240.97,331.69) .. controls (238.81,331.68) and (237.06,329.79) .. (237.07,327.47) -- cycle ;
    \draw   (271.15,326.75) .. controls (271.16,324.43) and (272.92,322.56) .. (275.08,322.57) .. controls (277.24,322.58) and (278.99,324.47) .. (278.98,326.79) .. controls (278.97,329.11) and (277.21,330.98) .. (275.05,330.97) .. controls (272.88,330.96) and (271.14,329.07) .. (271.15,326.75) -- cycle ;

    \draw    (257.22,265.44) -- (257.65,290.86) ;

    \draw    (243.03,323.75) -- (255.24,298.22) ;

    \draw    (272.8,323.07) -- (260.33,297.75) ;

    \draw   (301.15,293.08) .. controls (300.19,291.04) and (301.09,288.7) .. (303.16,287.86) .. controls (305.22,287.01) and (307.67,287.97) .. (308.62,290.01) .. controls (309.58,292.05) and (308.68,294.39) .. (306.61,295.24) .. controls (304.55,296.09) and (302.1,295.12) .. (301.15,293.08) -- cycle ;

    \draw   (298.79,328.15) .. controls (297.83,326.11) and (298.73,323.77) .. (300.8,322.93) .. controls (302.86,322.08) and (305.31,323.04) .. (306.26,325.08) .. controls (307.22,327.12) and (306.32,329.46) .. (304.25,330.31) .. controls (302.19,331.16) and (299.74,330.19) .. (298.79,328.15) -- cycle ;

    \draw   (330.91,314.01) .. controls (329.96,311.97) and (330.86,309.62) .. (332.92,308.78) .. controls (334.98,307.93) and (337.43,308.9) .. (338.39,310.94) .. controls (339.34,312.98) and (338.44,315.32) .. (336.38,316.16) .. controls (334.31,317.01) and (331.87,316.04) .. (330.91,314.01) -- cycle ;

    \draw    (292.18,265.71) -- (303.16,287.85) ;

    \draw    (302.92,322.52) -- (303.92,295.27) ;

    \draw    (330.96,310.12) -- (308.57,292.84) ;

    \draw   (323.14,269.06) .. controls (320.82,268.7) and (319.21,266.7) .. (319.55,264.58) .. controls (319.88,262.47) and (322.02,261.03) .. (324.34,261.39) .. controls (326.65,261.74) and (328.26,263.74) .. (327.93,265.86) .. controls (327.6,267.98) and (325.45,269.41) .. (323.14,269.06) -- cycle ;

    \draw   (353.06,290.16) .. controls (350.75,289.8) and (349.14,287.8) .. (349.47,285.68) .. controls (349.8,283.57) and (351.95,282.14) .. (354.26,282.49) .. controls (356.58,282.84) and (358.19,284.84) .. (357.86,286.96) .. controls (357.52,289.08) and (355.38,290.51) .. (353.06,290.16) -- cycle ;

    \draw   (357.39,256.67) .. controls (355.07,256.32) and (353.46,254.32) .. (353.79,252.2) .. controls (354.13,250.08) and (356.27,248.65) .. (358.59,249.01) .. controls (360.9,249.36) and (362.51,251.36) .. (362.18,253.48) .. controls (361.85,255.6) and (359.7,257.03) .. (357.39,256.67) -- cycle ;

    \draw    (294.09,261.24) -- (319.55,264.58) ;

    \draw    (350.23,283.76) -- (326.54,268.03) ;

    \draw    (353.95,254.51) -- (326.83,262.98) ;

    \draw  [fill={rgb, 255:red, 0; green, 0; blue, 255 }  ,fill opacity=1 ] (312.79,134.82) .. controls (314.95,135.44) and (316.23,137.69) .. (315.65,139.85) .. controls (315.06,142) and (312.84,143.25) .. (310.68,142.63) .. controls (308.52,142.01) and (307.25,139.76) .. (307.83,137.6) .. controls (308.41,135.44) and (310.63,134.2) .. (312.79,134.82) -- cycle ;

    \draw  [fill={rgb, 255:red, 0; green, 0; blue, 255 }  ,fill opacity=1 ] (225.85,139.89) .. controls (228.01,140.51) and (229.29,142.76) .. (228.71,144.92) .. controls (228.13,147.08) and (225.91,148.32) .. (223.75,147.7) .. controls (221.59,147.08) and (220.31,144.83) .. (220.89,142.67) .. controls (221.47,140.51) and (223.7,139.27) .. (225.85,139.89) -- cycle ;

    \draw  [fill={rgb, 255:red, 0; green, 0; blue, 255 }  ,fill opacity=1 ] (398.46,160.35) .. controls (400.62,160.97) and (401.9,163.22) .. (401.32,165.38) .. controls (400.74,167.54) and (398.52,168.78) .. (396.36,168.16) .. controls (394.2,167.54) and (392.92,165.29) .. (393.5,163.13) .. controls (394.08,160.97) and (396.31,159.73) .. (398.46,160.35) -- cycle ;

    \draw  [fill={rgb, 255:red, 0; green, 0; blue, 255 }  ,fill opacity=1 ] (366.06,190.31) .. controls (368.22,190.93) and (369.5,193.19) .. (368.92,195.34) .. controls (368.34,197.5) and (366.12,198.75) .. (363.96,198.13) .. controls (361.8,197.5) and (360.52,195.25) .. (361.1,193.1) .. controls (361.68,190.94) and (363.9,189.69) .. (366.06,190.31) -- cycle ;

    \draw  [fill={rgb, 255:red, 0; green, 0; blue, 255 }  ,fill opacity=1 ] (272.9,225.53) .. controls (275.06,226.15) and (276.34,228.4) .. (275.75,230.55) .. controls (275.17,232.71) and (272.95,233.96) .. (270.79,233.34) .. controls (268.63,232.72) and (267.35,230.46) .. (267.94,228.31) .. controls (268.52,226.15) and (270.74,224.9) .. (272.9,225.53) -- cycle ;

    \draw  [fill={rgb, 255:red, 0; green, 0; blue, 255 }  ,fill opacity=1 ] (158.77,210.2) .. controls (160.93,210.82) and (162.21,213.07) .. (161.63,215.23) .. controls (161.05,217.39) and (158.83,218.63) .. (156.67,218.01) .. controls (154.51,217.39) and (153.23,215.14) .. (153.81,212.98) .. controls (154.39,210.82) and (156.61,209.58) .. (158.77,210.2) -- cycle ;

    \draw  [fill={rgb, 255:red, 0; green, 0; blue, 255 }  ,fill opacity=1 ] (403.34,97.67) .. controls (405.5,98.29) and (406.77,100.54) .. (406.19,102.7) .. controls (405.61,104.85) and (403.39,106.1) .. (401.23,105.48) .. controls (399.07,104.86) and (397.79,102.61) .. (398.38,100.45) .. controls (398.96,98.29) and (401.18,97.05) .. (403.34,97.67) -- cycle ;

    \draw  [fill={rgb, 255:red, 0; green, 0; blue, 255 }  ,fill opacity=1 ] (182.13,237.58) .. controls (184.29,238.2) and (185.57,240.45) .. (184.99,242.61) .. controls (184.41,244.76) and (182.19,246.01) .. (180.03,245.39) .. controls (177.87,244.77) and (176.59,242.52) .. (177.17,240.36) .. controls (177.75,238.2) and (179.97,236.96) .. (182.13,237.58) -- cycle ;
 
     \draw  [fill={rgb, 255:red, 0; green, 0; blue, 255 }  ,fill opacity=1 ] (425.39,126.26) .. controls (427.55,126.88) and (428.83,129.13) .. (428.25,131.28) .. controls (427.67,133.44) and (425.45,134.69) .. (423.29,134.07) .. controls (421.13,133.45) and (419.85,131.19) .. (420.43,129.04) .. controls (421.01,126.88) and (423.23,125.63) .. (425.39,126.26) -- cycle ;
    
    \draw (405.53,77.3) node [anchor=north west][inner sep=0.75pt]    {$[ L_{2}]$};
    \draw (116.52,203.67) node [anchor=north west][inner sep=0.75pt]    {$[ L_{1}]$};
    \draw (282.15,219.57) node [anchor=north west][inner sep=0.75pt]    {$[ L_{3}]$};
    \end{tikzpicture}
\end{center}
\caption{
The bolytrope 
$\B_1(Q) = \B_1(
{\tiny \left(\begin{array}{@{}c@{}c@{}} 0 & 7 \\ 0 & 0 \end{array} \right)}
)$ 
in the Bruhat Tits tree of $\mathrm{SL}_2(\QQ_2)$. 
The green segment is the central polytrope $Q:=Q({\tiny\left(\begin{array}{@{}c@{}c@{}} 0 & 7 \\ 0 & 0 \end{array} \right)})  = \{ [L_{(i,0)} \mid 0\leq i \leq 7 \} $. 
The set ${\mathcal L} = Q({\tiny\left(\begin{array}{@{}c@{}c@{}} 0 & 8 \\ 1 & 0 \end{array} \right)}) $ is the convex hull of  $[L_1] = [L_{(0,1)}]$ and $[L_2] = [L_{(8,0)}]$. The blue vertices are the points at distance $1$ from $Q$. The PZ-order of the lattice classes $[L_1],[L_2]$ and $[L_3]$ is the same as the PZ-order of all the colored vertices.
}\label{fig:example_d=2}
\end{figure}

\begin{corollary}
    The $\PZ$-closed subset of $\mathcal{B}_2^0(K)$ are precisely the bolytropes. 
\end{corollary}
\begin{corollary}
The degree of a closed order $\Lambda $ in $K^{2\times 2}$ is $0$, $1$, or $2$.
Orders of degree $0$ are the maximal orders, whereas the closed orders of degree $1$ are precisely the graduated  non-maximal orders. 
All non-graduated closed orders in $K^{2\times 2} $ have degree $2$.  
\end{corollary}

\begin{remark}
Theorem \ref{thm:bolystar} implies \cite[Theorems 1 and 8]{Tu/11}. To see this, note that, by taking $[L_1], [L_2], [L_3]$ as in the proof of \Cref{thm:d=2}, we get that 
	$\Lambda = \PZ ([L_1],[L_2],[L_3])$.
\end{remark}

    \bigskip 
	
	{\bf Acknowledgements}.
	This project is supported by the Deutsche Forschungsgemeinschaft 
	(DFG, German Research Foundation) -- Project-ID 286237555 --TRR 195.
	We thank Bernd Sturmfels and Marvin Anas Hahn for mathematical discussions. We thank the anonymous referee for their valuable comments.

\end{document}